\documentclass[11pt]{article}
\usepackage{amsmath,amsthm,amsfonts}

\newcommand{\Y}{\hat{Y}}
\renewcommand{\=}{\doteq}

\newtheorem{thm}{Theorem}[section]
 \newtheorem{cor}[thm]{Corollary}
 \newtheorem{prop}[thm]{Proposition}

\theoremstyle{definition}
 \newtheorem{defn}[thm]{Definition}
\theoremstyle{definition}
 \newtheorem{rem}[thm]{Remark}
\numberwithin{equation}{section}
\begin{document} 
\title{\bf Moufang symmetry VI.\\Reductivity and hidden associativity in\\ Mal'tsev algebras}
\author{Eugen Paal}
\date{}
\maketitle

\thispagestyle{empty}

\begin{abstract}
Reductivity in the  Ma'tsev algebras is inquired. This property relates the Mal'tsev algebras to the general Lie triple systems.
\end{abstract}

\section{Introduction}
In the present paper we inquire reductivity in the Ma'tsev algebras. This property relates the Mal'tsev algebras to the general Lie triple systems. The paper can be seen as a continuation of \cite{Paal1,Paal2,Paal3,Paal4, Paal5}.

Throughout the paper we denote by $\Gamma$ the tangent algebra of a local analytic Moufang loop $G$.
Multiplication in $\Gamma$ is denoted by $[\cdot,\cdot]$.

\section{Yamaguti brackets}

In the tangent Mal'tsev algebra $\Gamma$ with binary brackets $[\cdot,\cdot]$  the ternary 
\emph{\hbox{Yamaguti} brackets} $[\cdot,\cdot,\cdot]$ are given \cite{Paal4} by
\begin{subequations}
\label{yam_brackets}
\begin{align}
[x,y,z]
&=[x,[y,z]]-[y,[x,z]]+[[x,y],z]\\
&=-[y,x,z]
\end{align}
\end{subequations}
The Sagle-Yamaguti identity reads \cite{Paal4}
\begin{equation}
\label{sagle-yamaguti}
[x,y,[z,w]]=[[x,y,z],w]+[z,[x,y,w]]
\end{equation}
We know from \cite{Paal5} that the infinitesimal translations of a local analytic Moufang loop $G$ have the triple closure property
\begin{equation}
\label{triple_closure}
3[[L_x,L_y],L_z]=L_{[x,y,z]+[[x,y],z]}
\end{equation}

\begin{prop}
The Yamaguti brackets satisfy relation
\begin{equation}
\label{glts2}
[x,y,z]+[y,z,x]+[z,x,y]+[[x,y],z]+[[y,z],x]+[[z,x],y]=0
\end{equation}
\end{prop}

\begin{proof}
Use the Jacobi identity
\begin{equation*}
[[L_x,L_y],L_z]+[[L_y,L_z],L_x]+[[L_z,L_x],L_y]=0
\end{equation*}
with the triple closure property (\ref{triple_closure}).
\end{proof}

We know from \cite{Paal4} that the Yamagutian $Y$ of $G$ satisfy the reductivity relation
\begin{equation}
\label{red_L}
6[Y(x;y),L_z]=L_{[x,y,z]}
\end{equation}
and the constraint
\begin{equation}
\label{yam_constraint}
Y([x;y],z)+Y([y;z],x)+Y([z;x],y)=0
\end{equation}

\begin{prop}
The Yamaguti brackets satisfy relation
\begin{equation}
\label{glts3}
[[x,y],z,u]+[[y,z],x,u]+[[z,x],y,u]=0
\end{equation}
\end{prop}

\begin{proof}
By using (\ref{red_L}) and (\ref{yam_constraint}) calculate
\begin{align*}
0
&=6[Y([x;y],z)+Y([y;z],x)+Y([z;x],y),L_u]\\
&=6[Y([x;y],z),L_u]+6[Y([y;z],x),L_u]+6[Y([z;x],y),L_u]\\
&=L_{[[x,y],z,u]}+L_{[[y,z],x,u]}+L_{[[z,x],y,u]}\\
&=L_{[[x,y],z,u]+[[y,z],x,u]+[[z,x],y,u]}
\end{align*}
which yields the desired relation.
\end{proof}

\section{Yamagutian in Mal'tsev algebra}

Define the left translations $l^+_z$ and the Yamagutian $\Y$ in $\Gamma$ by
\begin{align}
\label{yam_tan1}
l^+_xy&=[x,y]\\
6\Y(x;y)z&=[x,y,z]
\end{align}
\begin{prop}
We have
\begin{equation}
\label{l+}
6\Y(x;y)=[l^+_x,l^+_y]+l^+_{[x,y]}
\end{equation}
\end{prop}

\begin{proof}
Use (\ref{yam_brackets})
\end{proof}

\begin{rem}
In terms of K.~Yamaguti \cite{Yam63} one can say that the left translations  $l^+_x$ ($x\in\Gamma$) realize a \emph{weak representation} of the tangent Mal'tsev algebra $\Gamma$ of $G$.
\end{rem}

\begin{prop}
The Yamagutian $\Y$ obeys relations
\begin{gather*}
\Y(x;y)=-\Y(y,x)\\
\Y([x,y];z)+\Y([y,z];x)+\Y([z,x];y)=0
\end{gather*}
\end{prop}

\begin{proof}
Use (\ref{yam_brackets}b) and (\ref{glts3}).
\end{proof}

\section{Reductivity in Mal'tsev algebra}

\begin{thm}
Yamagutian $\Y$ is a derivation of $\Gamma$:
\begin{equation}
\label{yam_der}
\Y(x;y)[z,w]=[\Y(x;y)z,w]+[z,\Y(x;y)w]
\end{equation}
\end{thm}

\begin{proof}
Use the Sagle-Yamaguti identity (\ref{sagle-yamaguti}).
\end{proof}

\begin{thm}
Left translations of $\Gamma$ satisfy the reductivity relations
\begin{equation}
\label{red_gamma}
6[\Y(x;y)],l^+_z]=l^+_{[x,y,z]}
\end{equation}
\end{thm}

\begin{proof}
Rewrite (\ref{yam_der}) as
\begin{equation*}
\Y(x;y)l^+_z w=l^+_{\Y(x;y)z}w+l^+_z\Y(x;y)w
\end{equation*}
Thus, since $w$ in $\Gamma$ is arbitray, we have 
\begin{equation*}
[\Y(x;y),l^+_z]=l^+_{\Y(x;y)z}
\end{equation*}
which imply the desired relation (\ref{red_gamma}).
\end{proof}

\section{Hidden associativity in Mal'tsev algebra}

\begin{thm}[hidden associativity]
The Yamagutian $\Y$ of $G$ obey the commutation relations
\begin{equation}
\label{hidden_assoc1}
6[\Y(x;y),\Y(z,w)]=\Y([x,y,x],w)+\Y(z;[x,y,w])
\end{equation}
if  the Sagle-Yamaguti identity (\ref{sagle-yamaguti}) holds.
\end{thm}

\begin{proof}
We calculate the Lie bracket $[\Y(x;y),\Y(z,w)]$ from the Jacobi identity
\begin{equation}
\label{jacobi-temp-ad}
[[\Y(x;y),l^+_z],l^+_w]+[[l^+_z,l^+_w],\Y(x;y]+[[l^+_w,\Y(x;y),l^+_z]=0
\end{equation}
and formulae (\ref{l+}). We have
\begin{align*}
6[[\Y(x;y),l^+_z],l^+_w]
&=[l^+_{[x,y,z]},l_w]\\
&=6\Y([x,y,z];w)-l^+_{[[x,y,z],w]}\\
6[[l^+_z,l^+_w],\Y(x;y]
&=36[\Y(z;w),\Y(x,y)]-6[l^+_{[z,w]},\Y(x;y)]\\
&=36[\Y(z;w),\Y(x,y)]-l^+_{[x,y,[z,w]]}\\
6[[l^+_w,\Y(x;y),l^+_z]
&=6\Y(z;[x,y,w])-l^+_{[z,[x,y,w]]}
\end{align*}
Substituting these relations into (\ref{jacobi-temp-ad}) we obtain
\begin{align*}
36[\Y(x;y),\Y(z,w)]-6\Y([x,y,x],w)-6Y(z;[x,y,w])
&=l^+_{[x,y,[z,w]]-[[x,y,z],w]-[z,[x,y,w]]}\\
&=0
\tag*{\qed}
\end{align*}
\renewcommand{\qed}{}
\end{proof}

\begin{rem}
A.~Sagle \cite{Sagle} and K.~Yamaguti proved \cite{Yam62} that the identity (\ref{sagle-yamaguti}) is equivalent to the Mal'tsev identity. In terms of K.~Yamaguti \cite{Yam63} one can say that the Yamagutian $\Y$ is a \emph{generalized representation} of the tangent Mal'tsev algebra $\Gamma$ of $G$.
\end{rem}

\begin{thm}[hidden associativity]
The hidden associtivity property (\ref{hidden_assoc1}) is equivalent to relations
\begin{equation}
\label{hidden_assoc2}
[x,y,[z,w,v]]=[[x,y,z],w,v]+[z,[x,y,w],v]+[z,w,[x,y,v]]
\end{equation}
\end{thm}

\begin{proof}{I.}
It is enough to note that
\begin{align*}
36\Y(x;y)\Y(z;w)v&=[[x,y,z],w,v]\\
36\Y(z;w)\Y(x;y)v&=[z,w,[x,y,v]]\\
6\Y([x,y,z];w)v&=[[x,y,z],w,v]\\
6\Y(x;[x,y,w])v&=[z,[x,y,w],v]
\end{align*}
which imply the desired relation.
\end{proof}

\begin{proof}{II.}
Use the Jacobi identity
\begin{equation*}
[[Y(x;y),Y(z;w)],L_v]+[[Y(z;w),L_v],Y(x;y)]+[[L_v,Y(x;y],)Y(z;w)]=0
\end{equation*}
Note that
\begin{align*}
36[[Y(x;y),Y(z;w)],L_v]
&=6[Y([x,y,z];w)+Y(z;[x,y,w]),L_v]\\
&=L_{[[x,y,z],w,v]+[z,[x,y,w],v]}\\
36[[Y(z;w),L_v],Y(x;y)]
&=6[L_{[z,w,v]},Y(x;y)]\\
&=-L_{[x,y,[z,w,v]]}\\
36[[L_v,Y(x;y],)Y(z;w)]
&=L_{[z,w,[x,y,v]]}
\end{align*}
Adding the latter we obtain
\begin{align*}
0
&=36[Y(x;y),Y(z;w)],L_v]+[[Y(z;w),L_v],Y(x;y)]+[[L_v,Y(x;y],)Y(z;w)]\\
&=L_{-[x,y,[z,w,v]]+[[x,y,z],w,v]+[z,[x,y,w],v]+[z,w,[x,y,v]]}
\end{align*}
which yields the desired relations.
\end{proof}

\begin{cor}
Relation (\ref{hidden_assoc2}) means that the Yamagutian $\Y$ is a derivation of the Yamaguti brackets:
\begin{equation*}
\Y(x,y)[z,w,v]=[\Y(x,y)z,w,v]+[z,\Y(x,y)w,v]+[z,w,\Y(x,y)v]
\end{equation*}
\end{cor}

\begin{rem}[\cite{Yam63}]
It turns out that every derivation of the Mal'tsev algebra is a derivation of the Yamaguti brackets. The Yamagutian $\Y$ is called the \emph{inner derivation} of the Mal'tsev algebra.
\end{rem}

\section{Recapitulation: general Lie triple systems}

\begin{defn}[general Lie triple system]
A \emph{general Lie triple system} (GLTS) is a vector space $M$  with a binary brackets $[\cdot,\cdot]$ and a ternary brackets $[\cdot,\cdot,\cdot]$ that satisfy the following identities:
\begin{subequations}
\label{glts}
\begin{gather}
[x,y]=-[y,x]\\
[x,y,z]=-[y,x,z]\\
[x,y,z]+[y,z,x]+[z,x,y]+[[x,y],z]+[[y,z],x]+[[z,x],y]=0\\
[[x,y],z,u]+[[y,z],x,u]+[[z,x],y,u]=0\\
[x,y,[z,u]]=[[x,y,z],u]+[z,[x,y,u]]=0\\
[x,y,[z,w,v]]=[[x,y,z],w,v]+[z,[x,y,w],v]+[z,w,[x,y,v]]
\end{gather}
\end{subequations}
\end{defn}

\begin{thm}[see also  \cite{Yam63}]
The tangent algebra $\Gamma\=\{T_e(G),[\cdot,\cdot]\}$ of a local analytic Moufang loop $G$ is  a general Lie triple system with the Yamaguti brackets given by (\ref{yam_brackets}a).
\end{thm}

\begin{proof}
Relation (\ref{glts}a) is evident,  
(\ref{glts}b) coincides with  (\ref{yam_brackets}b), 
(\ref{glts}c) coincides with  (\ref{glts2}), 
(\ref{glts}d) coincides with  (\ref{glts3}), 
(\ref{glts}e) is the Sagle-Yamaguti identity (\ref{sagle-yamaguti}), 
(\ref{glts}f) coincides with  (\ref{hidden_assoc2}), 
\end{proof}

\begin{rem}
The general Lie triple systems turn out to be the tangent algebras of the \emph{reductive homogeneous spaces} \cite{Yam58}.
\end{rem}

\section{Hidden associativity: recapitulation}

\begin{thm}[hidden associativity]
The Yamagutian $Y$ satisfy the Jacobi identity 
\begin{equation*}
[[Y(x;y),Y(z;w)],L_v]+[[Y(z;w),L_v],Y(x;y)]+[[L_v,Y(z;w)],Y(x;y)]]=0
\end{equation*}
if  and only if (\ref{glts}f) holds.
\end{thm}

\begin{proof}
\begin{align*}
36[[Y(x;y)\Y(z;w)],L_v]
&=6[Y(x,y,z);w]+Y(z;[x,y,w],v)\\
&=L_{[[x,y,z],w,v]+[z,[x,y,w],v]}\\
36[[Y(z;w),L_v],Y(x;y)]
&=6L_{[z,w,v],Y(x;y)}\\
&=-L_{[x,y,[z,w,v]]}\\
36[[L_v,Y(z;w)],\Y(x;y)]]
&=L_{[z,w,[x,y,v]]}
\end{align*}
Adding these formulae we obtain
\begin{align*}
[[Y(x;y),Y(z;w)],L_v]
&+[[Y(z;w),L_v],Y(x;y)]+[[L_v,Y(z;w)],Y(x;y)]]=\\
&=L_{-[x,y,[z,w,v]]+[[x,y,z],w,v]+[z,[x,y,w],v]+[z,w,[x,y,v]]}
\tag*{\qed}
\end{align*}
\renewcommand{\qed}{}
\end{proof}

Just repeating the above proof we can propose
\begin{thm}[hidden associativity]
The Yamagutians $Y$ and $\Y$ satisfy the Jacobi identities 
\begin{align*}
[[Y(x;y),Y(z;w)],Y(u;v)]+[[Y(z;w),Y(u;v)],Y(x;y)]+[[Y(u;v),Y(z;w)],Y(x;y)]]&=0\\
[[\Y(x;y),\Y(z;w)],\Y(u;v)]+[[\Y(z;w),\Y(u;v)],\Y(x;y)]+[[\Y(u;v),\Y(z;w)],\Y(x;y)]]&=0
\end{align*}
if  (\ref{glts}f) holds.
\end{thm}

\section*{Acknowledgement}

Research was in part supported by the Estonian Science Foundation, Grant 6912.

\bigskip\noindent
Department of Mathematics\\
Tallinn University of Technology\\
Ehitajate tee  5, 19086 Tallinn, Estonia\\
E-mail: eugen.paal@ttu.ee

\end{document}